\documentclass[12pt,a4paper]{article}

\usepackage{latexsym,amssymb,graphicx,amsmath}
\makeindex

\begin{document}
\title{\bf \Large On $r$-Equitable Coloring \\ of Complete Multipartite Graphs}

\author{Chih-Hung Yen\thanks{The corresponding author.
Supported in part by the National Science Council under grant NSC101-2115-M-415-004.}\\
\normalsize Department of Applied Mathematics\\
\normalsize National Chiayi University\\
\normalsize Chiayi 60004, Taiwan\\
\normalsize E-mail address: chyen@mail.ncyu.edu.tw}

\date{}
\maketitle

\newtheorem{define}{Definition}
\newtheorem{proposition}[define]{Proposition}
\newtheorem{theorem}[define]{Theorem}
\newtheorem{lemma}[define]{Lemma}
\newtheorem{corollary}[define]{Corollary}
\newtheorem{problem}[define]{Problem}
\newtheorem{conjecture}[define]{Conjecture}
%
%
\newenvironment{proof}{
\par
\noindent {\bf Proof.}\rm}%
{\mbox{}\hfill\rule{0.5em}{0.809em}\par}

\newenvironment{remark}{\par
\noindent {\bf Remark.}\rm}

\baselineskip=16pt
\parindent=0.8cm

\begin{abstract}
\noindent
Let $r \geqslant 0$ and $k \geqslant 1$ be integers.
We say that a graph $G$ has an $r$-equitable $k$-coloring
if there exists a proper $k$-coloring of $G$ such that the sizes of any two color classes differ by at most $r$.
The least $k$ such that a graph $G$ has an $r$-equitable $k$-coloring is denoted by $\chi_{r=} (G)$, and the least $n$ such that a graph $G$ has an $r$-equitable $k$-coloring for all $k \geqslant n$ is denoted by $\chi^*_{r=} (G)$. In this paper, we propose a necessary and sufficient condition for a complete multipartite graph $G$ to have an $r$-equitable $k$-coloring, and also give exact values of $\chi_{r=} (G)$ and $\chi^*_{r=} (G)$.

\end{abstract}

%
\section{Introduction}
%

A graph $G=(V,E)$ is composed of a nonempty vertex set $V$ and an edge set $E$.
All graphs we consider in this paper are presumed to be undirected, finite, loopless, and without
multiple edges. For a positive integer $k$, a (proper) {\em $k$-coloring} of a graph $G$ is a mapping $f: V
\rightarrow \{1,2,\ldots ,k\}$ such that adjacent vertices have different
images. The images $1,2,\ldots,k$ are called {\em colors} and the corresponding sets $\{u\in V \colon f(u)=1\},\{u\in V \colon f(u)=2\},\ldots, \{u\in V \colon f(u)=k\}$ are called {\em color classes}. Obviously, a color class is an independent set
whose size may be equal to zero in $G$.
And one color in a $k$-coloring of a graph $G$ is said to be {\em missing} if its corresponding color class is an empty set of size zero.
Moreover, a graph is {\em $k$-colorable} if it has a $k$-coloring.
The {\em chromatic number} of a graph $G$, written $\chi(G)$, is the least $k$ such that $G$ is $k$-colorable.

A $k$-coloring of a graph $G$ is said to be {\em equitable} provided that the sizes of any two color classes differ by at most one.
A graph $G$ is {\em equitably $k$-colorable} if $G$ has an equitable $k$-coloring.
The least $k$ such that a graph $G$ is equitably $k$-colorable is called the {\em equitable chromatic number} of $G$ and denoted by $\chi_{=}(G)$.
The notion of equitable colorability was first introduced by Meyer \cite{Meyer1973} in 1973.
His motivation came from the problem of assigning one of the six days of the work week to each garbage collection route. And so far,
quite a few results on equitable coloring of graphs have been obtained in the literature, see $[1,2,4$-$10]$.

Recently, Hertz and Ries \cite{Hertz2011} generalized the notion of equitable colorability. They said that a $k$-coloring of a graph $G$ is {\em $r$-equitable} for an integer $r \geqslant 0$ if the sizes of any two color classes differ by at most $r$. And a graph $G$ is {\em $r$-equitably $k$-colorable} if there exists an $r$-equitable $k$-coloring of $G$. The least $k$ such that a graph $G$ is $r$-equitably $k$-colorable is called the {\em $r$-equitable chromatic number} of $G$ and denoted by $\chi_{r=}(G)$. It is clear that an $r$-equitably $k$-colorable graph is certainly $(r+1)$-equitably $k$-colorable. Moreover, an equitably $k$-colorable graph is also $1$-equitably $k$-colorable, and vice versa. In fact, such a generalization is quite natural since many $k$-colorable graphs do not have equitable $k$-colorings.

Unlike proper colorings of graphs, an equitably (or $r$-equitably) $k$-colorable graph may not be equitably (or $r$-equitably) $(k+1)$-colorable. For example, the graph in Figure \ref{fig1}, denoted by $K_{3,3}$, is equitably $2$-colorable, yet it is not equitably $3$-colorable. Hence, we also have an interest in finding the least $n$ such that a graph $G$ is equitably (or $r$-equitably) $k$-colorable for all $k \geqslant n$,
called the {\em equitable} (or {\em $r$-equitable}) {\em chromatic threshold} of $G$ and denoted by $\chi^*_{=} (G)$ (or $\chi^*_{r=} (G)$). Note that $\chi^*_{0=} (G)$ does not exist for any graph $G$. Because a graph $G$ is not $0$-equitably $k$-colorable for any $k \geqslant |V(G)|+1$.

\bigskip

\begin{figure}[htb]
\begin{center}
\includegraphics[scale=0.75]{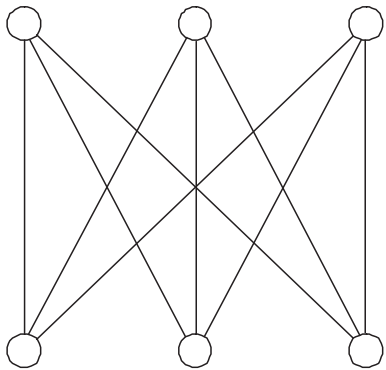}
\caption{The graph $K_{3,3}$.}\label{fig1}
\end{center}
\end{figure}

In this paper, we pay attention to $r$-equitable coloring of a particular class of graphs, called {\em complete multipartite graphs}. We first give a brief review for equitable coloring on complete multipartite graphs related to our results in this paper. Then, for any $r \geqslant 0$, we propose a necessary and sufficient condition for a complete multipartite graph $G$ to have an $r$-equitable $k$-coloring, and also give exact values of $\chi_{r=} (G)$ and $\chi^*_{r=} (G)$.

%
\section{Known results}
%

Recall that a graph $G$ is {\em $t$-partite} if its vertex set can be partitioned into $t$ independent sets $V_{1},V_{2},\ldots ,V_{t}$, and {\em complete $t$-partite}, denoted by $K_{n_{1},n_{2},\ldots,n_{t}}$, if every vertex in $V_{i}$ is adjacent to every vertex in $V_{j}$ whenever
$i \neq j$ and $|V_{i}|=n_{i} \geqslant 1$ for all $1 \leqslant i \leqslant t$. $V_{1},V_{2},\ldots ,V_{t}$ are called {\em partite sets} of $G$. By convention it is always assumed that $t \geqslant 2$ and $1 \leqslant n_1 \leqslant n_2 \leqslant \cdots \leqslant n_t$. And a graph is said to be {\em complete multipartite} if it is complete $t$-partite for some $t$. Furthermore, a complete $t$-partite graph $K_{n_{1},n_{2},\ldots,n_{t}}$ satisfies $n_{1}=n_{2}=\cdots =n_{t}=n$ is also denoted by $K_{t(n)}$.

Let $\lceil x \rceil$ and $\lfloor x \rfloor$ denote, respectively, the
smallest integer not less than $x$ and the largest integer not greater than $x$.
Also, let $\mathbb{N}$ denote the set of all positive integers. In 1994, Wu \cite{Wu1994} proved the followings.

\begin{theorem}\label{1} For any $K_{n_{1},n_{2},\ldots,n_{t}}$, let $p=n_1+n_2+\cdots+n_t$. Then
$K_{n_{1},n_{2},\ldots,n_{t}}$ is equitably $k$-colorable if and only if either $k > p$ or $n_i \geqslant
\lceil{n_i/\lceil p/k \rceil}\rceil \lfloor p/k \rfloor$ for all $1 \leqslant i \leqslant t$ and
$\sum^t_{i=1} \lfloor{n_i/\lfloor p/k \rfloor}\rfloor \geqslant k \geqslant \sum^t_{i=1}
\lceil{n_i/\lceil p/k \rceil}\rceil$ when $k \leqslant p$.
\end{theorem}

\begin{theorem}\label{2}
$\chi_{=}(K_{n_1,n_2,\ldots,n_t})=
\sum^t_{i=1} \lceil{n_i/h}\rceil$, where
$h= \max \{m \in \mathbb{N} \colon n_i \geqslant \lceil{n_i/m}\rceil (m-1)$ for all $1 \leqslant i \leqslant t\}$.
\end{theorem}

\begin{theorem}\label{3}
$\chi^{*}_{=}(K_{n_1,n_2,\ldots,n_t})=
\sum^t_{i=1} \lceil{n_i/h}\rceil$, where
$h= \min \{m \in \mathbb{N} \colon$there exists some $i$ such that $n_i < \lceil{n_i/(m+1)}\rceil m$ or there exist $n_i$ and $n_j$, $i\neq j$, such that both of $n_i$ and $n_j$ are not divisible by m$\}$.
\end{theorem}

Later, in 2001, Lam et al. \cite{Lam2001} also showed the following result which is equivalent to Theorem \ref{2}.

\begin{theorem}\label{4}
$\chi_{=}(K_{n_1,n_2,\ldots,n_t})=
\sum^t_{i=1} \lceil{n_i/(h+1)}\rceil$, where
$h= \max \{m \in \mathbb{N} \colon n_i \ ({\rm mod}
\ m)$ $< \lceil{n_i/m}\rceil$ for all $1 \leqslant i \leqslant t\}$.
\end{theorem}

Recently, in 2010, Lin and Chang \cite{Lin2010} showed the following results for $K_{t(n)}$.

\begin{theorem}\label{5} For any $k \geqslant t$, $K_{t(n)}$ is equitably $k$-colorable if and only if $\lceil{n/\lfloor{k/t}\rfloor}\rceil - \lfloor{n/\lceil{k/t}\rceil}\rfloor \leqslant 1$. \end{theorem}

\begin{theorem}\label{6} $\chi^*_{=}
(K_{t(n)}) = t \lceil{n/h}\rceil$, where $h$ is the least positive integer such that $n$ is not divisible by $h$. \end{theorem}

%
\section{Our results}
%

In what follows, let $I_n$ denote the graph consisting of $n$ isolated vertices, where $n \geqslant 1$.

\begin{lemma}\label{7} For any $r \geqslant 0$, $I_n$ has an $r$-equitable $k$-coloring if and only if there exists an integer $m \geqslant 0$ such that $(m+r)k \geqslant n \geqslant mk$. \end{lemma}

\begin{proof} $(\Rightarrow)$ Suppose that $I_n$ has an $r$-equitable $k$-coloring. Then there exists a $k$-coloring of $I_n$ such that each of the $k$ color classes is of size $m,m+1,\ldots,$ or $m+r$ for some integer $m \geqslant 0$.
Hence, we have $(m+r)k \geqslant n \geqslant mk$.

$(\Leftarrow)$ Firstly, since $n=\lceil{n/k}\rceil+
\lceil{(n-1)/k}\rceil+ \cdots +\lceil{(n-(k-1))/k}\rceil$, we partition the vertex set of
$I_n$ into $k$ independent sets $V_1,V_2,\ldots,V_k$ of sizes $\lceil{n/k}\rceil,
\lceil{(n-1)/k}\rceil,\ldots,$ $\lceil{(n-(k-1))/k}\rceil$,
respectively. Next, since there exists an integer $m \geqslant 0$ such that $(m+r)k \geqslant n \geqslant mk$,
we have $m+r \geqslant n/k \geqslant m$. It implies that $m+r \geqslant \lceil{n/k}\rceil \geqslant
\lfloor{n/k}\rfloor \geqslant m$ because $m+r$ and $m$ are integers. Then $I_n$ has a $k$-coloring such that each of the $k$ color classes is of size $m,m+1,\ldots,$ or $m+r$ by letting each of $V_1,V_2,\ldots,V_k$ be a color class and $m+r \geqslant \lceil{n/k}\rceil \geqslant \lceil{(n-1)/k}\rceil \geqslant \cdots \geqslant \lceil{(n-(k-1))/k}\rceil =\lfloor{n/k}\rfloor \geqslant m$. Hence, $I_n$ has an $r$-equitable $k$-coloring. \end{proof}

\begin{lemma}\label{8} For any $r \geqslant 1$, $K_{n_1,n_2,\ldots,n_t}$ has an $r$-equitable
$k$-coloring such that at least one color is missing if and only if $k \geqslant (\sum^t_{i=1}
\lceil{n_i/r}\rceil) +1$. \end{lemma}

\begin{proof} $(\Rightarrow)$ Clearly, if $K_{n_1,n_2,\ldots,n_t}$ has an $r$-equitable
$k$-coloring such that at least one color is missing, then there exists an $r$-equitable
$(k-1)$-coloring of $K_{n_1,n_2,\ldots,n_t}$ such that each of the $k-1$ color classes is of size $0,1,\ldots,$ or $r$. Hence, it implies that we can certainly find positive integers $k_1,k_2,\ldots,k_t$ such that $k-1=\sum^t_{i=1}k_i$ and $I_{n_i}$ has a $k_i$-coloring in which each of the $k_i$ color classes is of size $0,1,\ldots,$ or $r$ for all $1 \leqslant i \leqslant t$. Then we have $rk_i \geqslant n_i$ for all $1 \leqslant i \leqslant t$. Since $r \geqslant 1$ and $k_1,k_2,\ldots,k_t$ are positive integers,
$k_i \geqslant \lceil{n_i/r}\rceil$ for all $1 \leqslant i \leqslant t$.
Therefore, $k-1=\sum^t_{i=1}k_i \geqslant \sum^t_{i=1}\lceil{n_i/r}\rceil$ and thereby $k \geqslant (\sum^t_{i=1} \lceil{n_i/r}\rceil) +1$.

$(\Leftarrow)$ If $k \geqslant (\sum^t_{i=1} \lceil{n_i/r}\rceil)+1$, then $k-1 \geqslant \sum^t_{i=1}\lceil{n_i/r}\rceil$. Hence, we can certainly find positive integers $k_1,k_2,\ldots,k_t$ such that $k-1=\sum^t_{i=1}k_i$ and $k_i \geqslant \lceil{n_i/r}\rceil$ for all $1 \leqslant i \leqslant t$. So, $k_i \geqslant n_i/r$ and $rk_i \geqslant n_i \geqslant 1 > 0=0\cdot k_i$ for all $1 \leqslant i \leqslant t$. Then $I_{n_i}$ has a $k_i$-coloring such that each of the $k_i$ color classes is of size $0,1,\ldots,$ or $r$ for all $1 \leqslant i \leqslant t$ by the proof of Lemma \ref{7}. Therefore, there exists an $r$-equitable $(k-1)$-coloring of $K_{n_1,n_2,\ldots,n_t}$ such that each of the $k-1$ color class is of size $0,1,\ldots,$ or $r$ by $k-1=\sum^t_{i=1}k_i$. It implies that $K_{n_1,n_2,\ldots,n_t}$ has an $r$-equitable $k$-coloring such that at least one color is missing. \end{proof}

\bigskip

Note that $K_{n_1,n_2,\ldots,n_t}$ has no $0$-equitable
$k$-coloring such that at least one color is missing; otherwise, the order of $K_{n_1,n_2,\ldots,n_t}$ is equal to zero.

\begin{lemma}\label{9} For any $r \geqslant 0$, $K_{n_1,n_2,\ldots,n_t}$ has an $r$-equitable
$k$-coloring such that no color is missing if and only if there exists a positive integer $m$ such that $\lfloor n_i/m \rfloor \geqslant \lceil{n_i/(m+r)}\rceil$ for all $1 \leqslant i \leqslant t$ and $\sum^t_{i=1} \lfloor{n_i/m}\rfloor \geqslant k \geqslant \sum^t_{i=1}
\lceil{n_i/(m+r)}\rceil$. \end{lemma}

\begin{proof} $(\Rightarrow)$ It is obvious that if $K_{n_1,n_2,\ldots,n_t}$ has an $r$-equitable
$k$-coloring such that no color is missing, then we can certainly find positive integers $k_1,k_2,\ldots,k_t,$ and $m$ such that $k=\sum^t_{i=1}k_i$ and $I_{n_i}$ has a $k_i$-coloring in which each of the $k_i$ color classes is of size $m,m+1,\ldots,$ or $m+r$ for all $1 \leqslant i \leqslant t$. Hence, we have $(m+r)k_i \geqslant n_i \geqslant mk_i$ for all $1 \leqslant i \leqslant t$. Since $k_1,k_2,\ldots,k_t,$ and $m$ are positive integers,
it implies that $n_i/m \geqslant k_i \geqslant n_i/(m+r)$ and thereby $\lfloor{n_i/m}\rfloor \geqslant k_i \geqslant \lceil{n_i/(m+r)}\rceil$ for all $1 \leqslant i \leqslant t$.
Therefore, $\lfloor n_i/m \rfloor \geqslant
\lceil{n_i/(m+r)}\rceil$ for all $1 \leqslant i \leqslant t$ and
$\sum^t_{i=1} \lfloor{n_i/m}\rfloor \geqslant \sum^t_{i=1}k_i=k \geqslant \sum^t_{i=1}
\lceil{n_i/(m+r)}\rceil$.

$(\Leftarrow)$ If there exists a positive integer $m$ such that $\lfloor{n_i/m}\rfloor \geqslant \lceil{n_i/(m+r)}\rceil$ for all $1 \leqslant i \leqslant t$ and $\sum^t_{i=1} \lfloor{n_i/m}\rfloor \geqslant k \geqslant \sum^t_{i=1} \lceil{n_i/(m+r)}\rceil$, then we can certainly find positive integers $k_1,k_2,\ldots,k_t$ such that $k=\sum^t_{i=1}k_i$ and $\lfloor{n_i/m}\rfloor \geqslant k_i \geqslant \lceil{n_i/(m+r)}\rceil$ for
all $1 \leqslant i \leqslant t$. Hence, $n_i/m \geqslant k_i \geqslant n_i/(m+r)$ and thereby $(m+r)k_i \geqslant n_i \geqslant mk_i$ for all $1 \leqslant i \leqslant t$. Then $I_{n_i}$ has a $k_i$-coloring in which each of the $k_i$ color classes is of size $m,m+1,\ldots,$ or $m+r$ for all $1 \leqslant i \leqslant t$ by the proof of Lemma \ref{7}. Therefore, $K_{n_1,n_2,\ldots,n_t}$ has an $r$-equitable $k$-coloring such that no color is missing by $k=\sum^t_{i=1}k_i$ and $m \geqslant 1$. \end{proof}

\bigskip

By the conclusions of Lemmas \ref{8} and \ref{9},
we can conclude the necessary and sufficient condition
for a complete $t$-partite graph $K_{n_1,n_2,\ldots,n_t}$ to have an $r$-equitable $k$-coloring.

\begin{theorem}\label{10} For any $r \geqslant 0$, $K_{n_1,n_2,\ldots,n_t}$ has an $r$-equitable
$k$-coloring if and only if at least one of the following statements holds.
\begin{enumerate}
\item $r \geqslant 1$ and $k \geqslant (\sum^t_{i=1}\lceil{n_i/r}\rceil) +1$.
\item There exists a positive integer $m$ such that $\lfloor n_i/m \rfloor \geqslant \lceil{n_i/(m+r)}\rceil$ for all $1 \leqslant i \leqslant t$ and $\sum^t_{i=1} \lfloor{n_i/m}\rfloor \geqslant k \geqslant \sum^t_{i=1}\lceil{n_i/(m+r)}\rceil$. \end{enumerate} \end{theorem}

For example, $K_{3,5,7}$ has a $2$-equitable $k$-coloring such that at least one color is missing
if and only if $k \geqslant 10$. Moreover, if we choose $m=1,2,3$, then we get that $K_{3,5,7}$ has a $2$-equitable $k$-coloring such that no color is missing if and only if $15 \geqslant k \geqslant 4$.
Hence, $K_{3,5,7}$ has a $2$-equitable $k$-coloring if and only if $k \geqslant 4$.

\begin{theorem}\label{11} For any $r \geqslant 0$ and $1 \leqslant n_1 \leqslant n_2 \leqslant \cdots \leqslant n_t$,
let $\theta= \max \{ m \in \mathbb{N} \colon \lfloor n_i/m \rfloor \geqslant
\lceil{n_i/(m+r)}\rceil$ for all $1 \leqslant i \leqslant t\}$.
Then $\chi_{r=}(K_{n_1,n_2,\ldots,n_t})=\sum^t_{i=1} \lceil{n_i/(\theta+r)}\rceil$. \end{theorem}

\begin{proof} Firstly, since $\lfloor n_i/1 \rfloor = n_i \geqslant
\lceil{n_i/(1+r)}\rceil$ for all $1 \leqslant i \leqslant t$, we have that $\theta$ exists with $\theta \geqslant 1$. Secondly, if $m \geq n_1+1$, then $\lfloor n_1/m \rfloor =0 < 1 = \lceil{n_1/(m+r)}\rceil$. Hence, $\theta \leqslant n_1$. Finally, if $k = \sum^t_{i=1}
\lceil{n_i/(\theta+r)}\rceil$, then $K_{n_1,n_2,\ldots,n_t}$ has
an $r$-equitable $k$-coloring by the choice of $\theta$ and Theorem \ref{10}. Now, let $k < \sum^t_{i=1} \lceil{n_i/(\theta+r)}\rceil$, and suppose
that $K_{n_1,n_2,\ldots,n_t}$ has an $r$-equitable $k$-coloring. By $k < \sum^t_{i=1} \lceil{n_i/r}\rceil$ if $r \geqslant 1$ and Theorem \ref{10},
we know that there exists a positive integer $m$ such that $\lfloor n_i/m \rfloor \geqslant
\lceil{n_i/(m+r)}\rceil$ for all $1 \leqslant i \leqslant t$ and
$\sum^t_{i=1} \lfloor{n_i/m}\rfloor \geqslant k \geqslant \sum^t_{i=1}
\lceil{n_i/(m+r)}\rceil$. Then $m \leqslant \theta$ by the choice of $\theta$,
and thereby $k \geqslant \sum^t_{i=1} \lceil{n_i/(m+r)}\rceil \geqslant \sum^t_{i=1}
\lceil{n_i/(\theta+r)}\rceil$. It is a contradiction. Thus $K_{n_1,n_2,\ldots,n_t}$ has
no $r$-equitable $k$-coloring when $k < \sum^t_{i=1} \lceil{n_i/(\theta+r)}\rceil$. Therefore, we can conclude that $\chi_{r=}(K_{n_1,n_2,\ldots,n_t})=
\sum^t_{i=1} \lceil{n_i/(\theta+r)}\rceil$. \end{proof}

\begin{theorem}\label{12} For any $r \geqslant 1$ and $1 \leqslant n_1 \leqslant n_2 \leqslant \cdots \leqslant n_t$, let $m_1,m_2,\ldots,m_x$ be all positive integers such that $m_1 < m_2 < \cdots < m_x$ and $\lfloor n_i/m_j \rfloor \geqslant \lceil{n_i/(m_j+r)}\rceil$ for all $1 \leqslant i \leqslant t$ and $1 \leqslant j \leqslant x$. Also, let $M=\{m_1,m_2,\ldots,m_x\}$ and $\theta=\min\{m_j \in M \colon \sum^t_{i=1}\lceil{n_i/(m_j+r)}\rceil > (\sum^t_{i=1}\lfloor{n_i/m_{j+1}}\rfloor) +1$ or $m_j=m_x\}$. Then $\chi^*_{r=}
(K_{n_1,n_2,\ldots,n_t}) = \sum^t_{i=1} \lceil{n_i/(\theta+r)}\rceil$. \end{theorem}

\begin{proof} Firstly, since $\lfloor n_i/1 \rfloor = n_i \geqslant
\lceil{n_i/(1+r)}\rceil$ for all $1 \leqslant i \leqslant t$, we have $m_1=1$, and thereby $M$ is a nonempty set and $\theta$ exists with $\theta \geqslant 1$. Secondly, if $k \geqslant (\sum^t_{i=1}\lceil{n_i/r}\rceil) +1$, then $K_{n_1,n_2,\ldots,n_t}$ has an $r$-equitable $k$-coloring by Theorem \ref{10}. Finally, let $k$ satisfy $\sum^t_{i=1}\lceil{n_i/r}\rceil \geqslant k \geqslant \sum^t_{i=1}\lceil{n_i/(\theta+r)}\rceil$, and also let $m_{\ell}=\max\{m_j\in M \colon \sum^t_{i=1}\lfloor{n_i/m_j}\rfloor \geqslant k\}$. Since $\sum^t_{i=1}\lfloor{n_i/m_1}\rfloor =\sum^t_{i=1}n_i \geqslant \sum^t_{i=1}\lceil{n_i/r}\rceil \geqslant k$, we know that $m_{\ell}$ exists with $m_{\ell} \geqslant 1$. Also, $m_{\ell} \leqslant \theta$ by the choice of $\theta$. Then we want to show that $\sum^t_{i=1} \lfloor{n_i/m_{\ell}}\rfloor \geqslant k \geqslant \sum^t_{i=1}
\lceil{n_i/(m_{\ell}+r)}\rceil$, and thus $K_{n_1,n_2,\ldots,n_t}$ has an $r$-equitable $k$-coloring by Theorem \ref{10}. Suppose that $k < \sum^t_{i=1}
\lceil{n_i/(m_{\ell}+r)}\rceil$. Then $m_{\ell} < m_{\ell+1} \leqslant \theta$ by $k \geqslant \sum^t_{i=1}\lceil{n_i/(\theta+r)}\rceil$, and $k > \sum^t_{i=1}\lfloor{n_i/m_{\ell+1}}\rfloor$ by the choice of $m_{\ell}$. Hence, we have $\sum^t_{i=1}
\lceil{n_i/(m_{\ell}+r)}\rceil > (\sum^t_{i=1}\lfloor{n_i/m_{\ell+1}}\rfloor)+1$. It is a contradiction by $m_{\ell} < \theta$ and the choice of $\theta$. Now, let $k = (\sum^t_{i=1}\lceil{n_i/(\theta+r)}\rceil)-1$. Then $k < \sum^t_{i=1}\lceil{n_i/r}\rceil$ by $\theta \geqslant 1$. Also, $k < \sum^t_{i=1}\lceil{n_i/(m_j+r)}\rceil$ for each $m_j \leqslant \theta$. Moreover, by the choice of $\theta$, $k > \sum^t_{i=1}\lfloor{n_i/m_j}\rfloor$ for each $m_j > \theta$. So, there exists no $m_j \in M$ such that $\sum^t_{i=1} \lfloor{n_i/m_j}\rfloor \geqslant k \geqslant \sum^t_{i=1}
\lceil{n_i/(m_j+r)}\rceil$. Then $K_{n_1,n_2,\ldots,n_t}$ has no $r$-equitable $k$-coloring by Theorem \ref{10}. Thus we can conclude that $\chi^*_{r=}
(K_{n_1,n_2,\ldots,n_t}) = \sum^t_{i=1} \lceil{n_i/(\theta+r)}\rceil$. \end{proof}

\bigskip

In fact, it is not difficult to observe that if a graph $G$ has an $r$-equitable $k$-coloring such that at least one color is missing, then there must exist a positive integer $k' < k$ such that $G$ has an $r$-equitable $k'$-coloring in which no color is missing. Hence, the $r$-equitable chromatic number $\chi_{r=}(G)$ of a graph $G$ is actually equal to the least $k$ such that $G$ has an $r$-equitable $k$-coloring in which no color is missing. Similarly, the $r$-equitable chromatic threshold $\chi^*_{r=} (G)$ of a graph $G$ is actually equal to the least $n$ such that $G$ has an $r$-equitable $k$-coloring for all $k > n$ and $G$ has an $r$-equitable $n$-coloring in which no color is missing. Finally, according to the above theorems, we have the following corollaries.

\begin{corollary}\label{13}For any $r \geqslant 0$ and $k \geqslant t$, $K_{t(n)}$ has an $r$-equitable $k$-coloring if and only if $\lceil{n/\lfloor{k/t}\rfloor}\rceil - \lfloor{n/\lceil{k/t}\rceil}\rfloor \leqslant r$. \end{corollary}

\begin{proof} ($\Rightarrow$) Suppose that $K_{t(n)}$ has an $r$-equitable $k$-coloring. Then, either $r \geqslant 1$ and $k \geqslant t\lceil{n/r}\rceil +1$ or there exists a positive integer $m$ such that $t\lfloor{n/m}\rfloor \geqslant k \geqslant t\lceil{n/(m+r)}\rceil$ by Theorem \ref{10}. If $r \geqslant 1$ and $k \geqslant t\lceil{n/r}\rceil +1$, then $\lfloor k/t \rfloor \geqslant \lceil{n/r}\rceil \geqslant n/r$. Hence, we have $n/\lfloor{k/t}\rfloor \leqslant n/(n/r)=r$. Since $r$ is an integer, it implies that $\lceil{n/\lfloor{k/t}\rfloor}\rceil \leqslant r$ and thereby $\lceil{n/\lfloor{k/t}\rfloor}\rceil - \lfloor{n/\lceil{k/t}\rceil}\rfloor \leqslant r$. If there exists a positive integer $m$ such that $t\lfloor{n/m}\rfloor \geqslant k \geqslant t\lceil{n/(m+r)}\rceil$, then $n/m \geqslant \lfloor{n/m}\rfloor \geqslant \lceil{k/t}\rceil \geqslant k/t \geqslant \lfloor{k/t}\rfloor \geqslant \lceil{n/(m+r)}\rceil \geqslant n/(m+r)$.
Hence, we have $m \leqslant \lfloor{n/\lceil{k/t}\rceil}\rfloor \leqslant \lceil{n/\lfloor{k/t}\rfloor}\rceil \leqslant m+r$ because $m$ and $m+r$ are positive integers. It implies that $\lceil{n/\lfloor{k/t}\rfloor}\rceil - \lfloor{n/\lceil{k/t}\rceil}\rfloor \leqslant r$.

($\Leftarrow$) Let $m=\lfloor{n/\lceil{k/t}\rceil}\rfloor$. Firstly, if $m=0$, then $\lceil{k/t}\rceil \geqslant n+1$ and $\lfloor{k/t}\rfloor \geqslant n$. It implies that $r \geqslant \lceil{n/\lfloor{k/t}\rfloor}\rceil - \lfloor{n/\lceil{k/t}\rceil}\rfloor =1-0=1$ and $k/t > n$. Hence, $k \geqslant tn+1=t\lceil{n/1}\rceil +1 \geqslant t\lceil{n/r}\rceil +1$. Therefore, $K_{t(n)}$ has an $r$-equitable $k$-coloring by Theorem \ref{10}. Next, if $m \geqslant 1$, by $\lceil{n/\lfloor{k/t}\rfloor}\rceil - \lfloor{n/\lceil{k/t}\rceil}\rfloor \leqslant r$, then we have $m = \lfloor{n/\lceil{k/t}\rceil}\rfloor \leqslant n/\lceil{k/t}\rceil \leqslant n/(k/t) \leqslant n/\lfloor{k/t}\rfloor \leqslant \lceil{n/\lfloor{k/t}\rfloor}\rceil \leqslant m+r$. Hence, $n/m \geqslant \lfloor{n/m}\rfloor \geqslant \lceil{k/t}\rceil \geqslant k/t \geqslant \lfloor{k/t}\rfloor \geqslant \lceil{n/(m+r)}\rceil \geqslant n/(m+r)$. It implies that $\lfloor{n/m}\rfloor \geqslant \lceil{n/(m+r)}\rceil$ and $t\lfloor{n/m}\rfloor \geqslant k \geqslant t\lceil{n/(m+r)}\rceil$. Therefore, $K_{t(n)}$ has an $r$-equitable $k$-coloring by Theorem \ref{10}. \end{proof}

\begin{corollary}\label{14} For any $r \geqslant 0$, $\chi_{r=}(K_{t(n)}) = t=\chi_{=}(K_{t(n)})$. \end{corollary}

\begin{corollary}\label{15} For any $r \geqslant 1$ and $n \geqslant 1$, let $\theta$ be the least positive integer such that $\lfloor{n/(\theta+1)}\rfloor < \lceil{n/(\theta+r)}\rceil$. Then $\chi^*_{r=}(K_{t(n)}) = t \lceil{n/(\theta+r)}\rceil$. \end{corollary}

\begin{proof} Firstly, since $\lfloor{n/(n+1)}\rfloor=0 < 1=\lceil{n/(n+r)}\rceil$, we know that $\theta$ exists with $\theta \leqslant n$. Also, $\lfloor{n/m}\rfloor \geqslant \lfloor{n/(m+1)}\rfloor \geqslant \lceil{n/(m+r)}\rceil$ for each $m \in \{1,2,\ldots,\theta-1\}$ by the choice of $\theta$. Moreover, if $m=\theta -1$, then $\lfloor{n/(\theta-1+1)}\rfloor = \lfloor{n/\theta}\rfloor \geqslant \lceil{n/(\theta-1+r)}\rceil \geqslant \lceil{n/(\theta+r)}\rceil$. Hence, we have that $\lfloor{n/m}\rfloor \geqslant \lceil{n/(m+r)}\rceil$ for each $m \in \{1,2,\ldots,\theta\}$. Next, let $m_1,m_2,\ldots,m_x$ be all positive integers such that $m_1 < m_2 < \cdots < m_x$ and $\lfloor n/m_j \rfloor \geqslant \lceil{n/(m_j+r)}\rceil$ for all $1 \leqslant j \leqslant x$. Also, let $M=\{m_1,m_2,\ldots,m_x\}$. Then $m_1=1,m_2=2,\ldots,m_{\theta}=\theta$, and thereby $(\sum^t_{i=1}\lfloor{n/m_{j+1}}\rfloor)+1 > \sum^t_{i=1}\lfloor{n/m_{j+1}}\rfloor = t\lfloor{n/m_{j+1}}\rfloor \geqslant t\lceil{n/(m_j+r)}\rceil =\sum^t_{i=1}\lceil{n/(m_j+r)}\rceil$ for each $m_j \in \{m_1,m_2,\ldots,m_{\theta-1}\}$ by $\lfloor{n/(m+1)}\rfloor \geqslant \lceil{n/(m+r)}\rceil$ for each $m \in \{1,2,\ldots,\theta-1\}$. Furthermore, since $\lfloor{n/(\theta+1)}\rfloor < \lceil{n/(\theta+r)}\rceil$, it implies that $t\lfloor{n/(\theta+1)}\rfloor +t \leqslant t\lceil{n/(\theta+r)}\rceil$. Therefore, if $m_{\theta+1}$ exists, then $m_{\theta+1} \geqslant \theta+1$ and $(\sum^t_{i=1}\lfloor{n/m_{\theta+1}}\rfloor) +1 = t\lfloor{n/m_{\theta+1}}\rfloor +1 < t\lfloor{n/(\theta+1)}\rfloor +t \leqslant t\lceil{n/(\theta+r)}\rceil=\sum^t_{i=1}\lceil{n/(m_{\theta}+r)}\rceil$ for all $t \geqslant 2$. Thus we can conclude that $m_{\theta}=\min\{m_j \in M \colon \sum^t_{i=1}\lceil{n/(m_j+r)}\rceil > (\sum^t_{i=1}\lfloor{n/m_{j+1}}\rfloor) +1$ or $m_j=m_x\}$. Then $\chi^*_{r=}(K_{t(n)}) = t \lceil{n/(\theta+r)}\rceil$ by Theorem \ref{12} and $m_{\theta}=\theta$. \end{proof}

\section{Some concluding remarks}

The motivation for writing this paper was reading a paper titled ``on $r$-equitable colorings of trees and forests'' uploaded to the personal home page of Alain Heartz, see \cite{Hertz2011}. Although the notion of $r$-equitable colorability is a quite natural generalization, it seems to be proposed without precedent. Hence, we believe that such a paper might open the door for more interesting problems on equitable coloring of graphs in the future, and perhaps, for more valuable research. In this paper, we do some things on this side and view them as the beginning.

\bigskip

\noindent \textbf{Acknowledgments}

\bigskip

\noindent The author thanks the referees for many helpful comments which led to a better version of this paper.

%

\end{document}